\documentclass{cpc}

\usepackage{amsmath}
\usepackage{amssymb}
\usepackage{bm}
\usepackage{graphicx}
\usepackage{float}

\title[Constructing Families Of Cospectral Regular Graphs]
{Constructing Families Of Cospectral Regular Graphs}
\author[M. Haythorpe and A. Newcombe]{\spreadout{M. HAYTHORPE}$^1$ and \spreadout{A. NEWCOMBE}$^2$%
\thanks{Research partially supported by ARC Discovery Grant DP150100618.}\\
\affilskip {$^1$} Michael Haythorpe, Flinders University,\\
\affilskip 1284 South Road, Tonsley, SA Australia 5042\\
michael.haythorpe@flinders.edu.au\\
\affilskip {$^2$} Alex Newcombe, Flinders University,\\
\affilskip 1284 South Road, Tonsley, SA Australia 5042\\
alex.newcombe@flinders.edu.au}

\newtheorem{thm}{Theorem}[section]
\newtheorem{lem}[thm]{Lemma}
\newtheorem{defn}[thm]{Definition}
\newtheorem{conj}[thm]{Conjecture}
\newtheorem{ex}[thm]{Example}
\newtheorem{rem}[thm]{Remark}

\begin{document}
\maketitle

\begin{abstract}
A set of graphs are called cospectral if their adjacency matrices have the same characteristic polynomial.  In this paper we introduce a simple method for constructing infinite families of cospectral regular graphs.  The construction is valid for special cases of a property introduced by Schwenk.  For the case of cubic (3-regular) graphs, computational results are given which show that the construction generates a large proportion of the cubic graphs, which are cospectral with another cubic graph.
\end{abstract}

\section{Introduction}
The characteristic polynomial of a matrix $A$ is the polynomial in $x$, $\det (xI-A)$.  A set of simple graphs $\{G_1, G_2...,G_k\}$ are called \textit{cospectral} if their adjacency matrices have identical characteristic polynomials.  A graph is \textit{not uniquely determined by its spectrum} if there is at least one non-isomorphic graph with which it is cospectral.  The investigation of the prevalence and properties of such graphs is one of the classical open problems in spectral graph theory, and hence, methods for constructing cospectral graphs are of interest.  To date, such construction methods have mainly fallen into two categories.  The first involves performing various operations on the edges or vertices of graphs to produce new graphs which are cospectral with the initial graphs.  Perhaps the most well-known example of this is the switching method of Godsil and McKay which was introduced in \cite{godsil} and then later generalised in \cite{haem}.  The second category involves `pasting' graphs together in intelligent ways to achieve the cospectrality.  Recently, there has been some interest in the construction of cospectral graphs which provide control over certain graph properties.  In particular, cospectral graphs which are regular, among other properties, are investigated in \cite{bapat}, \cite{filar} and \cite{blaz}.  A common feature of the existing constructions which ensure regularity is the use of operations which are similar to a product of graphs.  However, the nature of these operations means that both the degree of regularity and the order of the resulting graphs can become very large, which may be undesirable.  In the present work, we introduce a construction for regular cospectral graphs which avoids these problems; indeed, arbitrarily large cospectral regular graphs of any desired degree can be constructed.   The construction is valid for special cases of the notion of removal cospectral vertices which was introduced by Schwenk \cite{schwenk} and studied further in \cite{row}.  We paraphrase Schwenk's definition as follows:

\begin{defn}
\label{def1}
For two graphs $G_1$ and $G_2$, the subsets of vertices $S \subset V(G_1)$ and $T \subset V(G_2)$ are called {\em removal cospectral} if there exists a bijection $f:S \rightarrow T$ such that for every subset $X \subseteq S$, the graphs $G_1 \setminus X$ and $G_2 \setminus f(X)$ are cospectral.
\end{defn}

Note that, since $\emptyset \subseteq S$, it is implicit in Definition \ref{def1} that $G_1$ and $G_2$ are cospectral. It was subsequently shown by Godsil \cite{godsil2} that the requirement that $G_1 \setminus X$ and $G_2 \setminus f(X)$ be cospectral for every subset $X \subseteq S$ can be replaced with the same requirement for only those subsets $X \subseteq S$ which have cardinality at most two, and the converse is also true. That is, $S$ and $T$ are removal cospectral if and only if, for all choices of $i,j \in S$, the graphs $G_1 \setminus \{i,j\}$ and $G_2 \setminus \{f(i),f(j)\}$ are cospectral.

Our present construction takes a set of cospectral (possibly isomorphic), $k$-regular graphs and uses them to construct a new set of cospectral graphs, in which the resulting graphs are still $k$-regular.  This involves selecting any choice of a $k$-regular graph to be combined with each graph in the former set.  The order of the newly constructed graphs depends on this selection, and the growth in order can be made very small if desired.  We then demonstrate that, for some small choices of order $N$, a large proportion of the cubic graphs which are not uniquely determined by their spectrum can be produced by this construction.  Interestingly, this proportion appears to be increasing with $N$.

\section{Preliminaries}
Throughout this manuscript we use standard graph theory notation such as can be found in \cite{cvet}.   All graphs used here are simple, connected and undirected.  The neighborhood of a vertex $v$ is the set of vertices adjacent to $v$ and is denoted by $N(v)$.  The adjacency matrix of a graph $G$ is denoted by $A(G)$.  A subgraph of $G$ arising by deleting a set of vertices $U\subset V(G)$ and all edges incident to those vertices is denoted by $G \setminus U$.  Edge deletions are denoted by $G - e$ where $e \in E(G)$.  We denote the characteristic polynomial of a graph $G$ by $\phi (G,x):=\det(xI-A(G))$.  The walk generating matrix of a graph is $W(G,x)=\sum_{r\geq0} A(G)^rx^r$ whose $(i,j)$-th entry, denoted by $W_{ij}(G,x)$, is the generating function for the set of all walks in $G$ from vertex $i$ to vertex $j$.  A walk starting at vertex $i$ and ending at vertex $j$ is an $i{\text -}j$ walk.\\

We next define a special case of Definition \ref{def1} and call vertices which satisfy this special case \textit{replaceable}.

\begin{defn}
\label{def2}
For two graphs $G_1$ and $G_2$, the vertices $u \in V(G_1)$ and $v \in V(G_2)$ are called {\em replaceable} if their respective neighborhoods $N(u)$ and $N(v)$ are removal cospectral. If the bijection defining the removal cospectral set is $g : N(u) \rightarrow N(v)$, then we denote the replaceable vertices by the tuple $(u,v,g)$.
\end{defn}

\begin{lem}
\label{lem1}
Let $G_1$ and $G_2$ be graphs with replaceable vertices $\big(u,v,g \big)$, where $u \in V(G_1)$ and $v \in V(G_2)$ and $g:N(u) \rightarrow N(v)$.  Then the set $N(u) \cup \{u\}$ is removal cospectral with $N(v) \cup \{v\}$, with the new removal cospectral bijection $f:N(u) \cup \{u\} \rightarrow N(v) \cup \{v\}$ being equal to $g$ on $N(u)$ and $f(u)=v$.
\end{lem}

\begin{proof}
Godsil proved in \cite{godsil2} that $S \subset V(G_1)$ and $T \subset V(G_2) $ are removal cospectral if and only if for all $i,j \in S$, the following equation holds for some bijection $h$.
\begin{equation} \label{eq1}
W_{ij}(G_1,x)=W_{h(i)h(j)}(G_2,x).
\end{equation}

In particular, if we set $S = N(u)$ and $T = N(v)$, then we have $h = g$. Then, (\ref{eq1}) holds if and only if, for every $r$, there is a bijection between $i{\text -}j$ walks of length $r$ in $G_1$ and $g(i){\text -}g(j)$ walks of length $r$ in $G_2$.  To complete the proof, we show that this correspondence implies that there are analagous bijections when $S = N(u) \cup \{u\}$.

For each $i,j \in N(u)$, let $w$ be any $i{\text -}j$ walk. Then $w$ can be uniquely extended to a walk from $u$ to $j$ by appending the edge $(u,i)$.  Also, $w$ can be uniquely extended to a $u{\text -}u$ walk by appending the edges $(u,i)$ and $(j,u)$.   These observations, along with the already existing correspondence between $i{\text -}j$ walks and $g(i){\text -}g(j)$ walks, provide one bijection between $u{\text -}j$ walks and $v{\text -}f(j)$ walks for each $j \in N(u)$, and another bijection between $u{\text -}u$ walks and $v{\text -}v$ walks. Hence, there is a bijection between walks in $G_1$ starting and ending in $N(u) \cup \{u\}$, and walks in $G_2$ starting and ending in $N(v) \cup \{v\}$.
\end{proof}

In Definition \ref{def3} we describe a type of graph composition which we call the \textit{vertex composition}.

\begin{defn}
\label{def3}
Let $G$ and $H$ be graphs with the vertices $u \in V(G)$ and $v \in V(H)$ and their neighborhoods $N(u)=\{u_1,u_2,...,u_k\}$, $N(v)=\{v_1,v_                                                                                                                                                                                                                                                                                                                                                                                                                                                                                                                                                                                                                                                                                                                                                                                                                                                                                                                                                                                                                                                                                                                                                                                       2,...,v_k\}$.  Define any bijection $f:N(u) \rightarrow N(v)$, then the graph $(G \circ H)$ is the graph with the vertex set $(G \setminus \{u\}) \cup (H \setminus \{v\})$ and the additional edges $(u_i,f(u_i))$; $i=1,2,...,k$.  Note that $|N(u)|=|N(v)|$.  This \textit{vertex composition} will be denoted by the tuple $\big( (G \circ H),u,v,f\big)$.
\end{defn}

Note that if $G$ and $H$ are both $k$-regular, then the graph $(G \circ H)$ will be as well.  An example of such a composition is displayed in Figure \ref{fig1}.

\begin{figure}[H]
\centering
\includegraphics[scale=0.3]{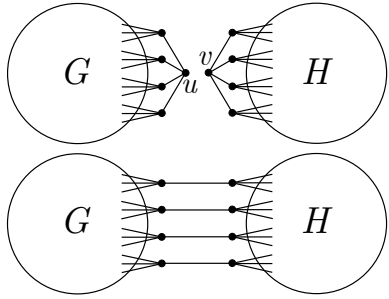}
\caption{Vertex composition $\big( (G \circ H),u,v,f\big)$.}
\label{fig1}
\end{figure}

We will be applying the vertex composition in Definition \ref{def3} upon sets of graphs, to produce new sets of graphs.  The new edges formed in each composition must be done so in a consistent manner.  To this end, consider a set of graphs $\{G_1,...,G_m\}$ with replaceable vertices $u_1 \in V(G_1),u_2 \in V(G_2),...,u_m \in V(G_m)$ such that $|N(u_i)|=k \ \forall i$.  Let $g_{ij}$ be the bijections defining the removal cospectral neighborhoods, hence we denote the replaceable vertices as the tuples $\big(u_i,u_j,g_{ij} \big)$.  Consider a fixed graph $H$ with a fixed vertex $h \in V(H)$, $|N(h)|=k$ and an arbirtrary bijection $f:N(u_1) \rightarrow N(h)$ which is to determine the new edges formed in the first composition $\big( (G_1 \circ H),u_1,h,f\big)$.  Then the new edges formed in the rest of the compositions are determined by  $\big( (G_i \circ H),u_i,h,f(g_{i1}) \big)$ for $i=2, \dots,m$.

\begin{thm}
\label{thm1}
Let $G_1$ and $G_2$ be graphs with replaceable vertices $\big(u,v,g\big)$ where $u \in V(G_1)$, $v \in V(G_2)$ and $g:N(u)\rightarrow N(v)$.  For a third graph $H$ with a fixed vertex $h \in V(H)$, the compositions $\big( (G_1 \circ H),u,h,f\big)$ and $\big( (G_2 \circ H),v,h,f(g^{-1}) \big)$ are cospectral.
\end{thm}

\begin{proof}
The method of proof follows from results of Godsil in \cite{godsil2} which are outlined below.  Consider two induced subgraphs within the graph $(G_1 \circ H)$, the first being the subgraph induced by the remaining vertices of $G_1$, and the second being the subgraph induced by $(H\cup N(u)):=F$.  Then plainly $G_1 \cap F = N(u)$.  The submatrix of the walk generating matrix $W(G_1,x)$ corresponding to the row and column indicies from $N(u)$ is denoted as $W_{N}(G_1,x)$.  Then Theorem $4.1$ of Godsil in \cite{godsil2} asserts that
\[W_{N}((G_1 \circ H),x)^{-1}=W_{N}(G_1\setminus u,x)^{-1}+W_{N}(F,x)^{-1}+xA(N(u))-I.\]
Taking the determinant of both sides, and then employing Theorem 3.1 from \cite{godsil2} on the left hand side, we see that
\[ \dfrac{x^{|N(u)|}\phi((G_1 \circ H),x^{-1})}{\phi(G_1 \setminus N(u),x^{-1}) \phi(H \setminus h,x^{-1})}=\det \big(W_{N}(G_1\setminus u,x)^{-1}+W_{N}(F,x)^{-1}+xA(N(u))-I\big).\]
As is described in \cite{godsil2}, the above shows that $\phi((G_1 \circ H),x^{-1})$ is determined by the polynomials $\phi(G_1 \setminus N(u),x^{-1})$, $ \phi(H \setminus h,x^{-1})$  and $\det \big(W_{N}(G_1\setminus u,x)^{-1}+W_{N}(F,x)^{-1}+xA(N(u))-I\big)$.  The terms inside the determinant expression are determined by $\phi(G_1 \setminus (K\cup \{u\}),x^{-1})$ and $\phi(F \setminus K,x^{-1})$ where $K$ ranges over all subsets of $N(u)$ of cardinality at most two.

By considering the same representation for the composition $\big( (G_2 \circ H),v,h,f(g^{-1}) \big)$, it is easily seen that the polynomials which determine $\phi((G_2 \circ H),x^{-1})$ are, by the cospectrality of $G_1$ and $G_2$ and by Lemma \ref{lem1}, the same as the polynomials which determine $\phi((G_1 \circ H),x^{-1})$.  Hence, the result follows.

\end{proof}

Next, similarly to the definitions above, another special case of Definition \ref{def1} is in regards to \textit{replaceable edges}.

\begin{defn}
\label{def4}
For two graphs $G_1$ and $G_2$, the edges $e_1 \in E(G_1)$ and $e_2 \in E(G_2)$ are called {\em replaceable} if their sets of incident vertices are removal cospectral. If the bijection defining the removal cospectral set is $g$, then we denote the replaceable edges by the tuple $(e_1,e_2,g)$.
\end{defn}

The following is a simple consequence of Definitions \ref{def1} and \ref{def4}.

\begin{lem}
\label{lem2}
Let $G_1$ and $G_2$ be graphs with replaceable edges $\big(e_1,e_2,g \big)$ where $e_1 \in E(G_1)$ and $e_2 \in E(G_2)$.  Then the vertices incident to $e_1$ remain removal cospectral to the vertices incident to $e_2$ in the graphs $G_1 - e_1$ and $G_2 - e_2$.
\end{lem}

\begin{proof}
Without loss of generality, let $e_1=(u_1,u_2)$, $e_2=(v_1,v_2)$ and $g(u_1)=v_1$, $g(u_2)=v_2$.  Note that deleting any one or two of the vertices $u_1$ and $u_2$ in $G_1 - e_1$ gives the same graph as the deletion in the original $G_1$.  So the only case which needs to be considered is when none of the vertices are deleted.  A well known representation of $\phi(G_1,x)$, e.g. see \cite{godsil2}, is
\begin{multline*}
\phi (G_1,x)= \phi (G_1-e_1,x)-\phi (G_1\setminus \{ u_1,u_2 \},x) \\
 - 2 \sqrt{\phi (G_1\setminus u_1,x)\phi (G_1\setminus u_2,x) - \phi (G_1,x)\phi (G_1\setminus \{u_1,u_2\},x)}.
\end{multline*}
The analogous representation of $\phi(G_2,x)$ is
\begin{multline*}
\phi (G_2,x)= \phi (G_2-e_2,x)-\phi (G_2\setminus \{ v_1,v_2 \},x) \\
 - 2 \sqrt{\phi (G_2\setminus v_1,x)\phi (G_2\setminus v_2,x) - \phi (G_2,x)\phi (G_2\setminus \{v_1,v_2\},x)}.
\end{multline*}
Comparing these two equations and considering cospectrality of the various vertex deleted graphs, reveals that $\phi(G_1-e_1,x)=\phi(G_2-e_2,x)$, and hence the result follows.

\end{proof}

In Definition \ref{def5}, we describe another type of graph composition which we call the \textit{edge composition}.

\begin{defn}
\label{def5}
Let $G$ and $H$ be graphs with the edges $e_1 \in E(G)$ and $e_2 \in E(H)$ and their respective incident vertices $u_1,u_2 \in V(G)$ and $v_1,v_2 \in V(H)$.  Define any bijection $f:\{u_1,u_2\} \rightarrow \{v_1,v_2\}$, then the graph $(G \diamond H)$ is the graph with the vertex set $V(G) \cup V(H)$, and edge set $E(G) \cup E(H) \setminus \{e_1,e_2\}$ plus the additional edges $(u_i,f(u_i)) \ i=1,2$.   This \textit{edge composition} will be denoted by the tuple $\big( (G \diamond H),e_1,e_2,f\big)$.  Figure \ref{fig2} is an illustration of $(G \diamond H)$.
\end{defn}
\begin{figure}[H]
\includegraphics[scale=0.3]{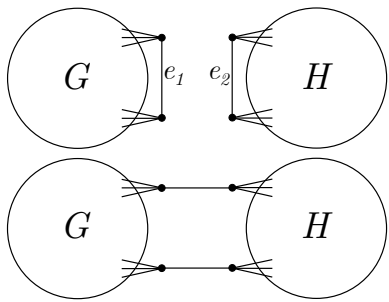}
\centering
\caption{Edge composition $\big( (G \diamond H),e_1,e_2,f\big)$ where $f(u_1)=v_1$ and $f(u_2)=v_2$.}
\label{fig2}
\end{figure}

By using Lemma \ref{lem2} and the analagous observations as in Theorem \ref{thm1}, the following can be shown.

\begin{thm}
\label{thm2}
Let $G_1$ and $G_2$ be graphs with replaceable edges $\big(e_1,e_2,g \big)$ where $e_1 \in E(G_1)$, $e_2 \in E(G_2)$, and $e_1 = (u_1,u_2)$.  For a third graph $H$ containing an edge $h \in E(H)$, where $h = (v_1,v_2)$, and an arbitrary bijection $f: \{u_1,u_2\} \rightarrow \{v_1,v_2\}$, the compositions $\big( (G_1 \diamond H),e_1,h,f\big)$ and $\big( (G_2 \diamond H),e_2,h,f(g^{-1})\big)$ are cospectral.
\end{thm}

\begin{rem}
\normalfont In the proceeding section we will discuss what happens when these constructions are applied to regular graphs.  In the special case when all graphs involved are 3-regular, the vertex composition and edge composition mimic, respectively, the types 3 and 2 breeding operations discussed in \cite{ban}.  From \cite{ban} we then know various properties of the resulting cospectral graphs.  For example, $(G \circ H)$ is planar or bipartite if and only if $G$ and $H$ both are, similarly for $(G \diamond H)$.  Also, $G$ and $H$ being Hamiltonian is a necessary condition for $(G \circ H)$ or $(G \diamond H)$ to be Hamiltonian.  The study of cospectral 3-regular graphs with differing Hamiltonicity is itself a topic of research, e.g. see \cite{bork,filar}.
\end{rem}


\section{Constructing cospectral regular graphs}
In the following, we only consider replaceable vertices, however an analogous method can be easily obtained for replaceable edges.  We begin with a set of cospectral $k$-regular graphs $\{G_1,...,G_m\}$, each of order $N$, with the vertices $u_1 \in V(G_1),u_2 \in V(G_2),...,u_m \in V(G_m)$, such that $(u_i,u_j,g_{ij})$ are replaceable vertices for any choice of graphs $G_i$ and $G_j$, and where $g_{ij} : N(u_i) \rightarrow N(u_j).$ We now illustrate how to use this set for the construction of new cospectral $k$-regular graphs.  Consider a second set of cospectral $k$-regular graphs $\{H_1,...,H_n\}$, of order $M$, defined similarly to above with replaceable vertices $v_i \in V(H_i)$, and let the bijections defining these removal cospectral sets be $h_{ij}:N(v_i) \rightarrow N(v_j)$.  Choose an arbitrary bijection $f:N(u_1) \rightarrow N(v_1)$ for the first composition $\big( (G_1 \circ H_1),u_1,v_1,f\big)$.  Then, by applying Theorem \ref{thm1} multiple times, all of the compositions $\big( (G_i \circ H_j),u_i,v_j,h_{1j}(f(g_{i1}))\big)$ for $ i=1,...,m$ and $j=1,...,n$ are cospectral.  This produces a set of cardinality $nm$ cospectral $k$-regular graphs on $N+M-2$ vertices. Note that choosing a different bijection $f$ could potentially produce an alternate set of cardinality $nm$ cospectral graphs. The tedious notation here should be cleared up upon viewing the example below.  Essentially we are just ensuring that the new edges are connected to the appropriate vertices.\\

\begin{ex}
\normalfont The cubic graphs $\{G_1, G_2\}$ in Figure \ref{Fig5} are cospectral with replaceable vertices $\big(t_1,u_1,g\big)$ and the cubic graphs $\{H_1, H_2\}$ in Figure \ref{Fig5} are also cospectral with replaceable vertices $\big(v_1,w_1,h \big)$.  The removal cospectral bijections are such that $g(t_i)=u_i$ and $h(v_i)=w_i$ for $i=2,3,4$.  First, consider the graph $H_1$ and an arbitrary bijection $f:N(t_1) \rightarrow N(v_1)$.  For this example we chose $f(t_i)=v_i$ for $i=2,3,4$.  Then the graph obtained from the composition $\big((G_1 \circ H_1),t_1,v_1,f\big)$ is the first graph displayed in Figure \ref{Fig6}.  The second graph displayed in Figure \ref{Fig6} is obtained from the composition $\big((G_2 \circ H_1),u_1,v_1,f(g^{-1})\big)$.  By Theorem \ref{thm1}, these two graphs are cospectral.  Next, we consider the graph $H_2$ and construct two more cospectral graphs with the compositions $\big((G_1 \circ H_2),t_1,w_1,h(f)\big)$ and $\big((G_2 \circ H_2),u_1,w_1,h(f(g^{-1}))\big)$, displayed as the final two graphs in Figure \ref{Fig6}. Then, by Theorem \ref{thm1}, we can conclude that all four graphs displayed in Figure \ref{Fig6} are cospectral.\\
\end{ex}

\begin{figure}[H]
\centering
\includegraphics[scale=0.35]{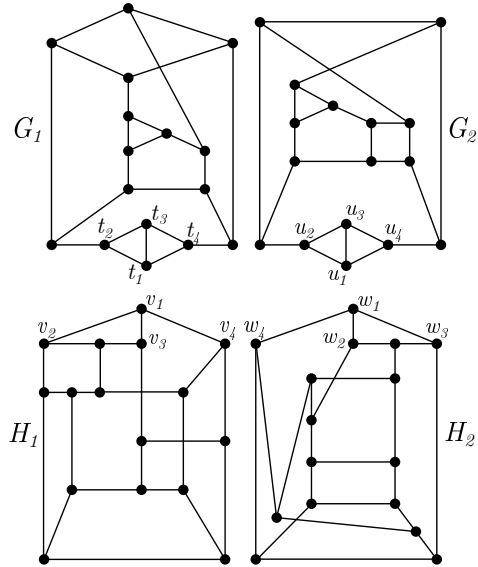}
\caption{The top left graph $G_1$ is cospectral with the top right graph $G_2$ with the replaceable vertices $\big(t_1,u_1,g\big)$. The bottom left graph $H_1$ is cospectral with the bottom right graph $H_2$ with the replaceable vertices $\big(v_1,w_1,h \big)$.}
\label{Fig5}
\end{figure}

\begin{figure}[H]
\centering
\includegraphics[scale=0.4]{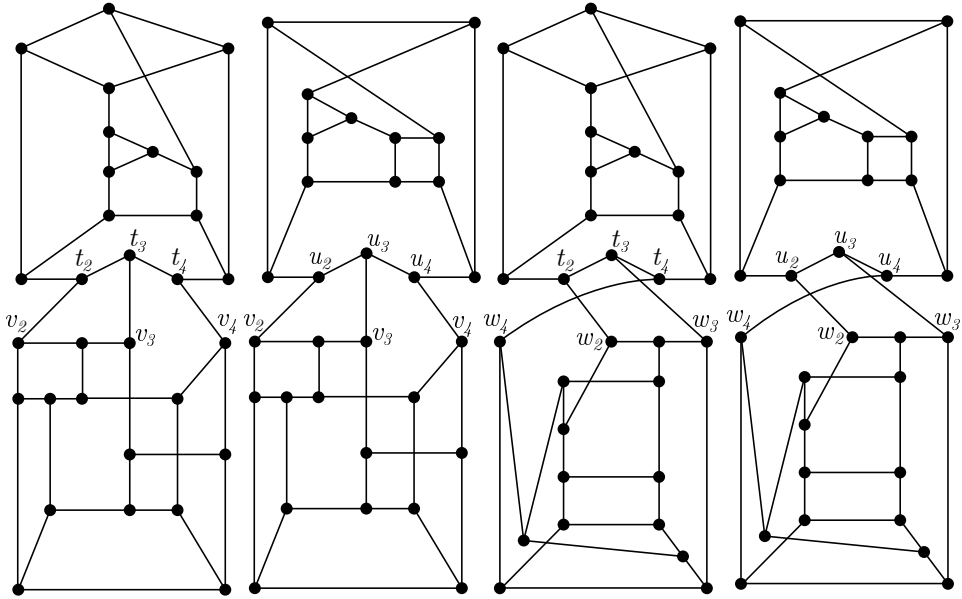}
\caption{The four non-isomorphic, cospectral, graphs arising from the compositions $\big((G_1 \circ H_1),t_1,v_1,f\big)$, $\big((G_2 \circ H_1),u_1,v_1,f(g^{-1})\big)$, $\big((G_1 \circ H_2),t_1,w_1,h(f)\big)$ and $\big((G_2 \circ H_2),u_1,w_1,h(f(g^{-1}))\big)$.}
\label{Fig6}
\end{figure}

We now make a couple of remarks about this method. First, it should be noted that cospectral $k$-regular graphs produced by this method have cyclic edge connectivity of at most $k$. Hence, any cospectral graphs with larger cyclic edge connectivity can not be produced in this manner. Second, the set of graphs $\{H_1, \dots, H_m\}$ could be chosen to have cardinality one. In such a case, any vertex in the one graph may be chosen as its ``replaceable" vertex.

\section{Computational results}
Cubic graphs provide a nice platform for conducting experiments because it is computationally tractable to perform exhaustive searches over the set of all cubic graphs of small order, say $N \leq 20$.  We provide various computational results in relation to replaceable vertices/edges and cubic graphs.  In Table \ref{tb1} we display the number of cubic graphs that possess replaceable vertices or edges within themselves, such that those vertices or edges lie in different orbits.

Let NUS3 (non-unique spectrum) denote a cubic graph which is cospectral with at least one other cubic graph, then in Table \ref{tb2} we demonstrate the commonness of replaceable vertices/edges among the cubic graphs which are NUS3.

\begin{table}[H]
\small
\centering
\caption{Total number of connected, non-isomorphic, cubic graphs of order $N$ and the numbers of those which contain replaceable vertices/edges within themselves, which are from different orbits.}
\label{tb1}
\begin{tabular}{| l l l l l l |}
\hline
Order & Total graphs & Contain rep. edge & \% & Contain rep. vertex & \%  \\
\hline
 12     &    85   	&    3   	& 3.6 	&  2 &  2.4 \\
 14     &  509  	&  16 		& 3.1 & 8  & 1.6\\
  16     &  4060  	&  115 	& 2.8 & 49 & 1.2 \\
   18     & 41301 	& 670   	& 1.6 & 354 &  0.9   \\
  20    &  510489   	& 4516 	& 0.9 &  1993 &  0.4   \\
\hline
\end{tabular}
\end{table}

\begin{table}[H]
\small
\centering
\caption{Number of cubic graphs of order $N$ which are cospectral with at least one other cubic graph (NUS3), and the numbers of those which also contain replaceable vertices/edges (with another NUS3 cubic graph).}
\label{tb2}
\begin{tabular}{ | l l l l l l | }
\hline
Order & NUS3 graphs & Contain rep. edge & \% & Contain rep. vertex & \%  \\
\hline
 14     &             6             &              6              & 100 &            4    &  66.7          \\
 16    &              83           &               77        &  92.8 &    65        &  78.3      \\
 18    &            956                &            868                & 90.8 &        800   &    83.7               \\
 20    &             9779               &           9529                 & 97.4 &               9271    &     94.8      \\
 22    &            114635                &         114304                   & 99.7 &     111325     &     97.1    \\
\hline
\end{tabular}
\end{table}

One possible explanation for the dramatic increase in commonness from Table \ref{tb1} to Table \ref{tb2} is that replaceable edges and vertices are often, in some sense, ``retained" when cospectral graphs are created.  This implies that we should be able to construct an increasing proportion of such graphs as the order increases.  This appears to be the case, as can be seen in Table \ref{tb3} which displays the number of cubic graphs which are NUS3 and the number of those which can be obtained as a result of the construction in Section 3. Recall that NUS3 graphs obtained by the construction in Section 3 have cyclic edge connectivity no larger than three; since these are the only graphs that can be obtained by our method, we also include the number of these in Table \ref{tb3}, and use NUS3C to denote such graphs.

\begin{table}[H]
\begin{center}
\small
\caption{Number of cubic graphs of order $N$ which are cospectral with at least one other cubic graph (NUS3), the number of such graphs with cyclic edge connectivity at most 3 (NUS3C), and the number of those which can be seen as the result of our construction. The proportion of NUS3 and NUS3C graphs that we generate is also given.}
\label{tb3}
\begin{tabular}{ | l l l l c c | }
\hline
Order & NUS3 graphs & NUS3C graphs & Number constructed & \% (NUS3) & \% (NUS3C)  \\
\hline
 14    &         6            &     6                &        4             &       66.7     &    66.7        \\
 16    &         83           &     65               &        40            &       48.2     &    61.5        \\
 18    &         956          &     841              &        492           &       51.5     &    58.5        \\
 20    &         9779         &     7604             &        6163          &       63.0     &    81.0        \\
 22    &         114635       &     89858            &        78775         &       68.7     &    87.7        \\
\hline
\end{tabular}
\end{center}
\end{table}

As a final remark, denote the set of all cubic graphs of order $N$ which are NUS3 as $C_N$, and then the subset of $C_N$ which consists of graphs which are produced by one of the constructions, denote as $C^*_N$.  Then Table \ref{tb3} suggests the following conjecture, which we leave untouched for future investigations.

\begin{conj} \label{cj1}
$$\lim_{N \rightarrow \infty} \dfrac{|C^*_N|}{|C_N|}=1.$$\\
\end{conj}


\begin{thebibliography}{99}
\bibitem{haem}
Abiad, A., Haemers, W.\,H. (2012)
Cospectral graphs and regular orthogonal matrices of level 2.
{\em Electron. J. Combin.} {\bf 19}, 13--29.

\bibitem{ban}
Baniasadi, P., Ejov, V., Filar, J.\,A., Haythorpe, M. (2016)
{\em Genetic theory for cubic graphs},
SpringerBriefs in Operations Research, Springer Publishing Company.

\bibitem{bapat}
Bapat, R.\,B., Karimi, M. (2016)
Construction of cospectral regular graphs,
{\em Mat. Vesnik} {\bf 68}, 66--76.

\bibitem{blaz}
Blazsik, Z.\,L., Cummings, J., Haemers, W.\,H. (2015)
Cospectral regular graphs with and without a perfect matching.
{\em Disc. Math.} {\bf 338}, 199--201.

\bibitem{bork}
Borkar, V.\,S., Ejov, V., Filar, J.\,A., Nguyen, G.\,T. (2012)
{\em Hamiltonian Cycle Problem and Markov Chains.},
Springer Science \& Business Media, pp. 7.

\bibitem{cvet}
Cvetkovic, D., Rowlinson, P., Simic, S. (1997)
{\em Eigenspaces of Graphs.},
Cambridge University Press.

\bibitem{filar}
Filar, J.\,A., Gupta, A., Lucas, S.\,K. (2005)
 Connected cospectral graphs are not necessarily both Hamiltonian.
{\em Aust. Math. Soc. Gaz.} {\bf 32}, 193.

\bibitem{godsil2}
Godsil, C.\,D. (1992)
Walk Generating Functions, Christophell-Darboux Identities and the Adjacency Matrix of a Graph.
{\em Combin. Probab. Comput.} {\bf 1}, 13--25.

\bibitem{godsil}
Godsil, C.\,D., Mckay, B.\,D. (1982)
Construction of cospectral graphs.
{\em Aequationes Math.} {\bf 25}, 257--268.

\bibitem{row}
Rowlinson, P. (1993)
The Characteristic Polynomials of Modified Graphs.
{\em Disc. App. Math.} {\bf 67}, 209--219.

\bibitem{schwenk}
Schwenk, A.\,J. (1979)
Removal-cospectral sets of vertices in a graph.
{\em Proc. 10th Southeastern International Conference on Combinatorics, Graph Theory \& Computing},
Utilitas Math., Winnipeg, Manitoba.

\end{thebibliography}
\end{document}